\documentclass[12pt]{amsart}

\def\boxit{$\sqcap\kern-8pt\sqcup$}
\def\littbox{\null\hfill\boxit{}}

\usepackage[a4paper,centering,margin=2.5cm]{geometry}
\usepackage{amssymb}
\usepackage{graphicx,psfrag,color}
\usepackage{enumerate}

\title{Ramsey for complete graphs with dropped cliques}
\author[J. Chappelon \and L.P. Montejano \and J.L. Ram\'irez Alfons\'in]{Jonathan Chappelon \and Luis Pedro Montejano \and Jorge Luis Ram\'irez Alfons\'in}
\thanks{The second author was supported by CONACYT}
\address{Universit\'{e} Montpellier 2, Institut de Math\'{e}matiques et de Mod\'{e}lisation de Montpellier, Case Courrier 051, Place Eug\`{e}ne Bataillon, 34095 Montpellier Cedex 05, France.}
\email{jonathan.chappelon@um2.fr}
\email{lpmontejano@gmail.com}
\email{jramirez@um2.fr}
\date{\today}

\theoremstyle{plain}
\newtheorem{theorem}{Theorem}[section]
\newtheorem{lemma}[theorem]{Lemma}
\newtheorem{proposition}[theorem]{Proposition}
\newtheorem{corollary}[theorem]{Corollary}

\theoremstyle{definition}

\newtheorem{conjecture}[theorem]{Conjecture}
\newtheorem{example}[theorem]{Exemple}

\newtheorem{remark}[theorem]{Remark}

\newtheorem{question}[theorem]{Question}

\renewcommand{\le}{\leqslant}

\renewcommand{\ge}{\geqslant}

\makeatletter
\g@addto@macro{\endabstract}{\@setabstract}
\newcommand{\authorfootnotes}{\renewcommand\thefootnote{\@fnsymbol\c@footnote}}%
\makeatother

\begin{document}

\begin{center}
\LARGE 
Ramsey for complete graphs with dropped cliques
\par\bigskip\normalsize\authorfootnotes
Jonathan Chappelon\footnote{Corresponding author}\footnote{E-mail address: jonathan.chappelon@um2.fr}\textsuperscript{1}, Luis Pedro Montejano\footnote{E-mail address: lpmontejano@gmail.com}\textsuperscript{1} and Jorge Luis Ram\'{i}rez Alfons\'{i}n\footnote{E-mail address: jramirez@um2.fr}\textsuperscript{1} \par\bigskip \textsuperscript{1}{Universit\'{e} Montpellier 2, Institut de Math\'{e}matiques et de Mod\'{e}lisation de Montpellier, Case Courrier 051, Place Eug\`{e}ne Bataillon, 34095 Montpellier Cedex 05, France} \par\bigskip
December 12, 2014
\end{center}

\begin{abstract}
Let $K_{[k,t]}$ be the complete  graph on $k$ vertices from which a set of edges, induced by a clique of order $t$, has been dropped. In this note we give two explicit upper bounds for  $R(K_{[k_1,t_1]},\dots, K_{[k_r,t_r]})$ (the smallest integer $n$ such that for any  $r$-edge coloring of $K_n$ there always occurs a monochromatic $K_{[k_i,t_i]}$ for some $i$).  Our first upper bound contains a classical one in the case when $k_1=\cdots =k_r$ and $t_i=1$ for all $i$.  The second one is obtained by introducing a new edge coloring called {\em $\chi_r$-colorings}. We finally discuss a conjecture claiming, in particular, that our second upper bound improves the classical one in infinitely many cases.\\[2ex]
\end{abstract}


\section{Introduction}

Let $K_n$ be a complete graph and let $r\ge 2$ be an integer. A $r$-edge coloring of a graph is a surjection from $E(G)$ to $\{0,\dots ,r-1\}$ (and thus each color class is not empty).  Let $k\ge t\ge 1$ be positive integers. We denote by $K_{[k,t]}$ the complete  graph on $k$ vertices from which a set of edges, induced by a clique of order $t$, has been dropped, see Figure \ref{fig3}. 

\begin{figure}[htb] 
\includegraphics[width=.4\textwidth]{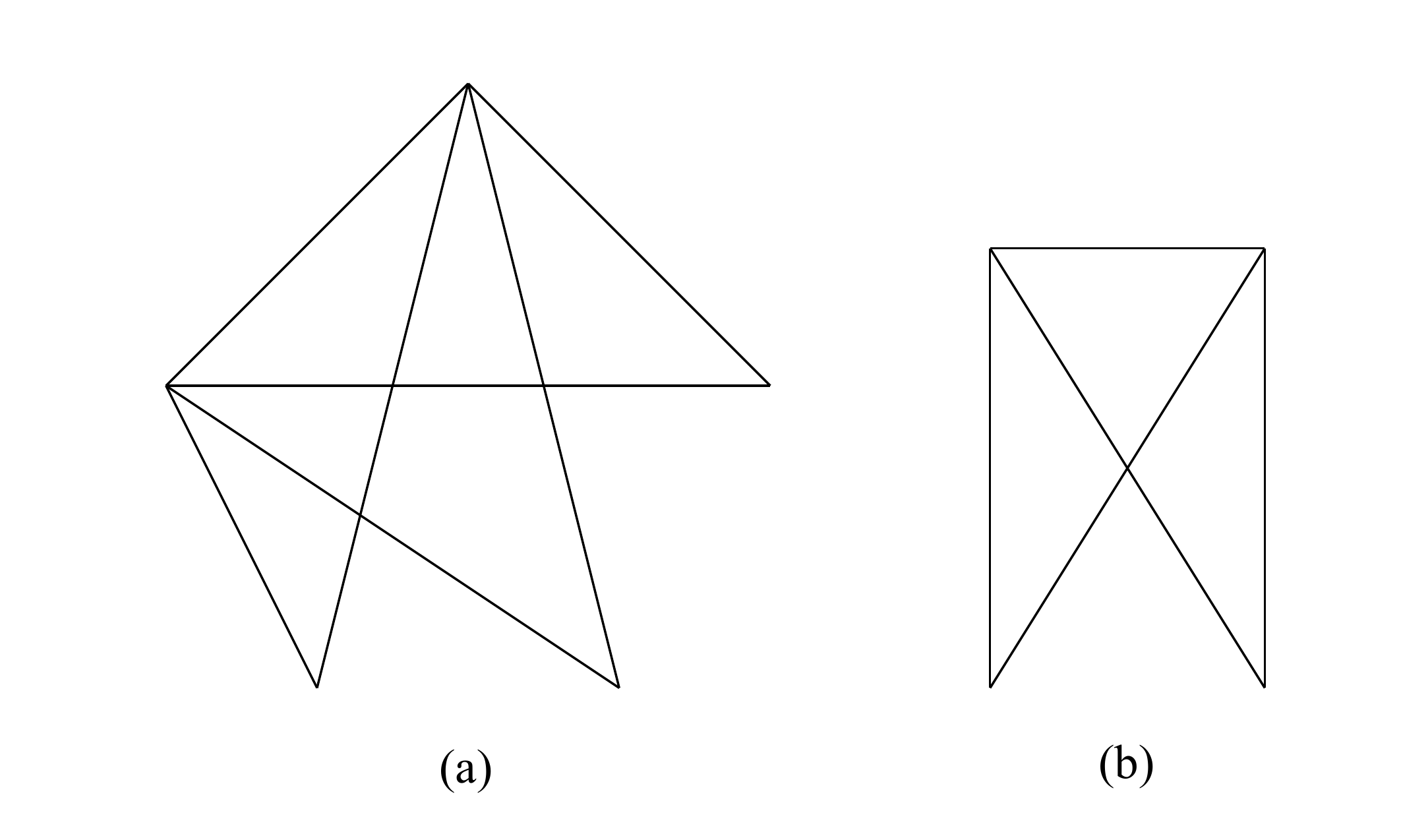}
\caption{(a) $K_{[5,3]}$ and (b) $K_{[4,2]}$}
\label{fig3}\end{figure}

Let $k_1,\dots ,k_r$ and $t_1,\ldots,t_r$ be positive integers with $k_i\ge t_i$ for all $i\in\{1,\ldots,r\}$.  Let $R([k_1,t_1],\dots, [k_r,t_r])$ be the smallest integer $n$ such that for any  $r$-edge coloring of $K_n$ there always occurs a monochromatic $K_{[k_i,t_i]}$ for some $i$.  

In the case when  $k_i=t_i$ for some $i$, we set 
$$
R([k_1,t_1],\dots, [k_{i-1},t_{i-1}],[t_i,t_i],[k_{i+1},t_{i+1}],\dots , [k_r,t_r])\le t_i.
$$ 
We note that equality is reached at $\min\limits_{1\le i\le r}\{t_i | t_i=k_i\}$.
Since the set of all the edges of $K_{[t_i,t_i]}$ (which is empty) can always be colored with color~$i$. We also notice that the case $R([k_1,1],\dots, [k_r,1])$ is exactly the classical Ramsey number $r(k_1,\dots ,k_r)$ (the smallest integer $n$ such that for any $r$-edge coloring of $K_n$ there always occurs a monochromatic $K_{k_i}$ for some $i$).  We refer the reader to the excellent survey \cite{Rad} on Ramsey numbers for small values.
\smallskip

In this note, we investigate general upper bounds for $R([k_1,t_1],\dots, [k_r,t_r])$.  In the next section we present a recursive formula that yields to an explicit general upper bound (Theorem \ref{mainth}) generalizing the well-known explicit upper bound due to Graham and R\"odl  \cite{GR} (see equation \ref{eq2}). We also improve our explicit upper bound  when $r=2$ for certain values of $k_i,t_i$ (Theorems~\ref{th2} and \ref{th3}).
\smallskip

In Section \ref{xi}, we shall present another general explicit upper bound for $R([k_1,t_1],\dots, [k_r,t_r])$ (Theorem \ref{mainth2}) by introducing a new edge coloring called {\em $\chi_r$-colorings}. We end by discussing a conjecture that
is supported by graphical and numerical results. 

\section{Upper bounds}

The following recursive inequality is classical in Ramsey theory
\begin{eqnarray}\label{eeqq}
r(k_1,k_2,\dots ,k_r) & \le r(k_1-1,k_2,\dots ,k_r)+r(k_1,k_2-1,\dots ,k_r)+\cdots +\\
& +r(k_1,k_2,\dots ,k_r-1)-(r-2) \nonumber
\end{eqnarray}

In the same spirit, we have the following.

\begin{lemma}\label{recur} Let $r\ge 2$ and let $k_1,\dots ,k_r$ and $t_1,\ldots,t_r$ be positive integers with $k_i\ge t_i+1\ge 2$ for all $i$. Then,
$$\begin{array}{ll}
R([k_1,t_1],\dots, [k_r,t_r]) & \le \ R([k_1-1,t_1],[k_2,t_2],\dots, [k_r,t_r])\\
& \ \ \ + R([k_1,t_1],[k_2-1,t_2],\dots, [k_r,t_r])\\
& \ \ \ \ \ \ \ \ \ \ \ \ \ \  \vdots \\
& \ \ \ + R([k_1,t_1],[k_2,t_2],\dots, [k_r-1,t_r])-(r-2).
\end{array}$$
\end{lemma}

A similar recursive inequality has been treated in \cite{SZ} in a more general setting (by considering 
a  family of graphs  intrinsically constructed via two operations {\em disjoin unions} and {\em joins}, see also \cite{HZ} for the case $r=2$). Although the latter could be used to obtain Lemma \ref{recur}, the arguments used here give a different and a more straight forward proof.
\smallskip

{\em Proof of Lemma \ref{recur}.} Let us take any $r$-edge coloring of $K_{N}$ with 
$$N\ge R([k_1-1,t_1],[k_2,t_2],\dots, [k_r,t_r])+\cdots + R([k_1,t_1],[k_2,t_2],\dots, [k_r-1,t_r])-(r-2).$$
Let $v$ a vertex of $K_N$ and let $\Gamma_i(v)$ be the set of all vertices joined to $v$ by an edge having color $i$ for each $i=1,\dots ,r$.
We claim that there exists index $1\le i\le r$ such that 
$$\Gamma_i(v)\ge R([k_1,t_1],\dots, [k_i-1,t_{i}],\dots, [k_r,t_r]).$$

Otherwise, 
$$\begin{array}{ll}
N-1=d(v)=\sum\limits_{j=1}^r\Gamma_j(v) & \le \sum\limits_{j=1}^r (R([k_1,{t_1}],\dots, [k_i-1,t_{i}],\dots, [k_r,t_r]) -1)\\
&  = \sum\limits_{j=1}^r (R([k_1,{t_1}],\dots, [k_i-1,t_{i}],\dots, [k_r,t_r])-r\\
& \le N + (r-2)-r= N-2
\end{array}$$
which is a contradiction. 
\medskip

Now, suppose that $\Gamma_i(v)\ge R([k_1,{t_1}],\dots, [k_i-1,t_{i}],\dots, [k_r,t_r])$ for an index $i$. By definition of 
$R([k_1,{t_1}],\dots, [k_i-1,t_{i}],\dots, [k_r,t_r])$ we have that the complete graph induced by $\Gamma_i(v)$ contains either 
a subset of vertices inducing a copy $K_{[k_j,t_j]}$ having all edges with color $j$, for some $j\neq i$, and we are done  or a subset of vertices inducing $K_{[k_i-1,t_i]}$ having all edges with color $i$. Adding vertex $v$ to $K_{[k_i-1,t_i]}$
we obtain the desired copy of $K_{[k_i,t_i]}$ having all edges colored with color  $i$.
\littbox

\subsection{Explicit general upper bound}
Lemma \ref{recur}  yield us to the following general upper bound for $R([k_1,t_1],\dots, [k_r,t_r])$. The latter was not treated in \cite{SZ} at all (in fact, suitable values/bounds needed to upper bound  the recursion given in \cite{SZ} for $R([k_1,t_1],\dots, [k_r,t_r])$ seem to be very difficult to estimate). 

\begin{theorem}\label{mainth}
Let $r\ge 2$ be a positive integer and let $k_1,\ldots,k_r$ and $t_1,\ldots,t_r$ be positive integers such that $k_i\ge t_i$ for all $i\in\{1,\ldots,r\}$. Then,
$$
R\left( [k_1,t_1] , \ldots , [k_r,t_r] \right) \le \max_{1\le i\le r}\{t_i\}\binom{k_1+\cdots+k_r-(t_1+\cdots+t_r)}{k_1-t_1,k_2-t_2,\ldots\ldots,k_r-t_r}
$$
where $\binom{n_1+n_2+\cdots+n_r}{n_1,n_2,\ldots\ldots,n_r}$ is the multinomial  coefficient defined by $\binom{n_1+n_2+\cdots+n_r}{n_1,n_2,\ldots\ldots,n_r}=\frac{(n_1+\cdots+n_r)!}{n_1!n_2!\cdots n_r!}$, for all nonnegative integers $n_1,\ldots,n_r$.
\end{theorem}

\begin{proof}
We suppose that $t_1,\ldots,t_r$ are fixed. We proceed by induction on $k_1+\cdots+k_r$, using Lemma~\ref{recur}. In the case where $k_j=t_j$, for some $j\in\{1,\ldots,r\}$, we already know that
$$
R\left( [k_1,t_1] , \ldots , [k_{j-1},t_{j-1}] , [t_j,t_j] , [k_{j+1},t_{j+1}] , \ldots , [k_r,t_r] \right) = t_j,
$$
and, since $k_i-t_i\ge 0$ for all $i$,
$$
\binom{k_1+\cdots+k_{i-1}+k_{i+1}+\cdots+k_r-(t_1+\cdots+t_{i-1}+t_{i+1}+\cdots+t_r)}{k_1-t_1,\ldots,k_{j-1}-t_{j-1},0,k_{j+1}-t_{j+1}\ldots\ldots,k_r-t_r} \ge 1.
$$
Therefore
$$
R\left( [k_1,t_1] , \ldots , [k_r,t_r] \right) = t_j \le \max_{1\le i\le r}{t_i}\binom{k_1+\cdots+k_r-(t_1+\cdots+t_r)}{k_1-t_1,k_2-t_2,\ldots\ldots,k_r-t_r}
$$
in this case. Now, suppose that $k_i>t_i$ for all $i\in\{1,\ldots,r\}$. By Lemma~\ref{recur} and by induction hypothesis, we obtain that
$$\begin{array}{ll}
R([k_1,t_1],\dots, [k_r,t_r]) & \le \ R([k_1-1,t_1],[k_2,t_2],\dots, [k_r,t_r])\\
& \ \ \ + R([k_1,t_1],[k_2-1,t_2],\dots, [k_r,t_r])\\
& \ \ \ \ \ \ \ \ \ \ \ \ \  \vdots \\
& \ \ \ + R([k_1,t_1],[k_2,t_2],\dots, [k_r-1,t_r])-(r-2) \\ \ \\
& \le \begin{array}[t]{l}
\displaystyle\max_{1\le i\le r}{t_i} \left( \binom{k_1+\cdots+k_r-(t_1+\cdots+t_r)-1}{k_1-t_1-1,k_2-t_2,\ldots\ldots,k_r-t_r} \right. \\ \ \\
+ \displaystyle\binom{k_1+\cdots+k_r-(t_1+\cdots+t_r)-1}{k_1-t_1-1,k_2-t_2-1,\ldots\ldots,k_r-t_r} \\ \ \\
\quad\quad\quad\vdots \\ \ \\
+ \left.\displaystyle\binom{k_1+\cdots+k_r-(t_1+\cdots+t_r)-1}{k_1-t_1-1,k_2-t_2,\ldots\ldots,k_r-t_r-1} \right) - (r-2)\\
\end{array} \\ \ \\
& \le \displaystyle\max_{1\le i\le r}{t_i}\binom{k_1+\cdots+k_r-(t_1+\cdots+t_r)}{k_1-t_1,k_2-t_2,\ldots\ldots,k_r-t_r},
\end{array}$$
since we have the following multinomial identity
$$
\binom{n_1+n_2+\cdots+n_r}{n_1,n_2,\ldots\ldots,n_r} = \sum_{i=1}^{r}\binom{n_1+n_2+\cdots+n_r-1}{n_1,\ldots,n_{i-1},n_i-1,n_{i+1},\ldots,n_r}
$$
for all positive integers $n_1,n_2,\ldots,n_r$.
\end{proof}

Theorem \ref{mainth} is a natural generalization of the only known explicit upper bound for classical Ramsey numbers. Indeed,
an immediate consequence of the above theorem (when $t=1$) is the following classical upper bound due to Graham and R\"odl  \cite[(2.48)]{GR} that was obtained by using \eqref{eeqq}.

\begin{equation}\label{eq2a}
R\left( [k_1,1] , \ldots , [k_r,1] \right) \le {{k_1+\cdots +k_r-r}\choose{k_1-1,\dots,k_r-1}}\cdot
\end{equation}

Let $R_r([k,t])=R(\underbrace{[k,t],\dots, [k,t]}_r)$.

\begin{corollary}\label{cor1a} Let $k\ge t\ge 2$ and $r\ge 2$ be integers. Then,
$$
R_r([k,t])\le t \binom{r(k-t)}{k-t,\ldots,k-t}.
$$
\end{corollary}

An immediate consequence of the above corollary (again when $t=1$) is the following upper bound
\begin{equation}\label{eq2}
R_r([k,1])\le \frac{(rk-r)!}{((k-1)!)^r}\cdot
\end{equation}

\subsection{Case $r=2$}
When $r=2$, it is the exact values of the recursive sequence generated from $u_{t,k} = u_{k,t} = t (= R_2([t,t]))$ for all $k\ge t$ and following the recursive identity $u_{k_1,k_2} = u_{k_1-1,k_2} + u_{k_1,k_2-1}$ for all $k_1,k_2\ge t+1$.
\smallskip

We investigate with more detail the cases $R([s,2],[t,2])$ (resp. $R([s,2],[t,1])$), that is, the smallest integer $n$ such that for any  $2$-edge coloring of $K_n$ there always occurs a  monochromatic $K_s-\{e\}$ or $K_t-\{e\}$ (resp. a  monochromatic $K_s-\{e\}$ or $K_t$)). These cases have been extensively studied and values/bounds for specific $s$ and $t$ are known, see Table \ref{current1} obtained from \cite{Rad}.

\begin{center}
\begin{table}[h] 
\centering
\footnotesize
\noindent\makebox[\textwidth]{%
\begin{tabular}{|c|c|c|c|c|c|c|c|c|c|}
 \hline
 & $K_3\setminus \{e\}$ & $K_4\setminus \{e\}$ &$K_5\setminus \{e\}$ &$K_6\setminus \{e\}$ &$K_7\setminus \{e\}$ &$K_8\setminus \{e\}$ &$K_9\setminus \{e\}$ &$K_{10}\setminus \{e\}$ &$K_{11}\setminus \{e\}$ \\
\hline
$K_3\setminus \{e\}$ & 3 & 5 & 7 & 9 & 11 & 13 & 15 & 17 & 19\\
\hline
$K_4\setminus \{e\}$ & 5 & 10 & 13 & 17 & 28 & [29,38] & 34 & 41 & \\
\hline
$K_5\setminus \{e\}$ & 7 & 13 & 22 & [31,39] & [40,66] & & & & \\
\hline
$K_6\setminus \{e\}$ & 9& 17 & [31,39] & [45,70] & [59,135] & & & &\\ 
\hline
$K_7\setminus \{e\}$ & 13 & 28 & [40,66] & [59,135] & 251 & & & & \\ 
\hline
$K_3$ & 5 & 7 & 11 & 17 & 21 & 25 & 31 & 37 & [42,45]\\
\hline
$K_4$ & 7 & 11 & 19 & [30,33] & [37,52] & 75 & 105 & 139 & 184\\
\hline
$K_5$ & 9 & 16 & [30,34] & [43,67] & 112 & 183 &277 &409 &581 \\
\hline
$K_6$ & 11 & 21 & [37,53] & 110 & 205 &373 &621 &1007 &1544 \\ 
\hline
$K_7$ & 13 & [28,30] & [51,83] & 193 & 392 &753 &1336 &2303 &3751 \\ 
\hline
$K_8$ & 15 & 42 & 123 & 300 & 657 &1349 &2558 &4722 &8200 \\ 
\hline          
\end{tabular}}
\vspace{.2cm}
\caption{\label{current1}Known bounds and values of $R([s,2],[t,2])$ and $R([s,2],[t,1])$.}
\end{table}
\end{center}

Lemma \ref{recur} allows to give (old) and new upper bounds for infinitely many cases.

\begin{theorem}\label{th2}
\begin{enumerate}[(a)]
\item[]
\item\cite[3.1 (a)]{Rad}
$R([3,2],[k,2])=2k-3$ for all $k\ge 2$,
\item
$R([4,2],[k,2])\le k^2-2k-39$ for all $k\ge 10$,
\item
$R([5,2],[8,2])\le 104$ and $R([5,2],[k,2])\le \frac{1}{3}k^3 - \frac{1}{2}k^2 - \frac{239}{6}k + 294$ for all $k\ge 9$,
\item
$R([6,2],[k,2])\le \frac{1}{12}k^4 - \frac{241}{12}k^2 + 274k - 1009$ for all $k\ge 8$,
\item
$R([7,2],[k,2])\le \frac{1}{60}k^5 + \frac{1}{24}k^4 - \frac{20}{3}k^3 + \frac{3047}{24}k^2 - \frac{17507}{20}k + 2064$ for all $k\ge 7$.
\end{enumerate}
\end{theorem}

\begin{proof}
\begin{enumerate}[(a)]
\item[]
\item
The result is obvious for $k=2$. First, let us show that $R([3,2],[3,2])=3$. For, we notice that $K_{[3,2]}$ is the graph consisting of three vertices, one of degree 2 and two of degree 1, and so $R([3,2],[3,2])>2$. Now, for any 2-coloring of the edges of $K_3$ there is always a vertex with two incident edges with the same color, giving the desired $K_{[3,2]}$.

Suppose now that $k\ge4$. We first prove that $R([3,2],[k,2])\le 2k-3$. For, we iterate inequality of Lemma \ref{recur} obtaining
$$\begin{array}{ll}
R([3,2],[k,2]) & \le R([2,2],[k,2]) + R([3,2],[k-1,2])\\
& =2 +R([3,2],[k-1,2]) \\
& \le 2+ R([2,2],[k-1,2]) + R([3,2],[k-2,2])\\
& = 2+ 2 + R([3,2],[k-2,2])\\
& \le \cdots \le \underbrace{2+\cdots +2}_{k-3}+R([3,2],[3,2])\\
& =2(k-3)+3=2k-3.
\end{array}$$

Now, we show that $R([3,2],[k,2])>2k-4$. For, take a perfect matching of $K_{2(k-2)}$. We color the edges belonging to the matching in red and all others in blue. We have neither a red $K_{[3,2]}$ red (since there are not vertex with two incident edges in red) nor  a blue $K_{[k,2]}$ since any subset of $k$ vertices forces to have at least two red edges.
\item[]
\item
It is known \cite{CH} that $R([4,2],[10,2])\le41$.  By using the latter and the recurrence of Lemma \ref{recur}, we obtain
$$
\begin{array}{ll}
R([4,2],[k,2]) & \le \sum\limits_{i=11}^k R([3,2],[i,2]) + R([4,2],[10,2])\\
& \le \sum\limits_{i=11}^k (2i-3) + 41= k^2-2k-39,
\end{array}
$$
for all integers $k\ge 11$.
\item
It is known \cite{HZ} that $R([5,2],[7,2])\le 66$. By using the latter and the recurrence of Lemma \ref{recur}, we obtain
$$
R([5,2],[8,2]) \le R([4,2],[8,2]) + R([5,2],[7,2]) \le 38 + 66 = 104,
$$
$$
R([5,2],[9,2]) \le R([4,2],[9,2]) + R([5,2],[8,2]) \le 34 + 104 = 138,
$$
and, for all integers $k\ge 10$,
$$
R([5,2],[k,2]) \begin{array}[t]{l}
\le \displaystyle\sum_{i=10}^{k}R([4,2],[i,2]) + R([5,2],[9,2]) \\
\le \displaystyle\sum_{i=10}^{k}(i^2-2i-39) + 138 \\ \ \\
= \displaystyle\frac{1}{3}k^3 - \frac{1}{2}k^2 - \frac{239}{6}k + 294.
\end{array}
$$
\item
It is known \cite{HZ} that $R([6,2],[7,2])\le 135$. By using the latter and the recurrence of Lemma \ref{recur}, we obtain
$$
R([6,2],[8,2]) \le R([5,2],[8,2]) + R([6,2],[7,2]) \le 104 + 135 = 239,
$$
and, for all integers $k\ge 9$,
$$
R([6,2],[k,2]) \begin{array}[t]{l}
\le \displaystyle\sum_{i=9}^{k}R([5,2],[i,2]) + R([6,2],[8,2]) \\
\le \displaystyle\sum_{i=9}^{k}\left(\frac{1}{3}i^3 - \frac{1}{2}i^2 - \frac{239}{6}i + 294\right) + 239 \\ \ \\
= \displaystyle\frac{1}{12}k^4 - \frac{241}{12}k^2 + 274k - 1009.
\end{array}
$$
\item
It is known \cite{SZ2} that $R([7,2],[7,2])\le 251$. By using the latter and the recurrence of Lemma \ref{recur}, we obtain for all integers $k\ge 8$,
$$
R([7,2],[k,2]) \begin{array}[t]{l}
\le \displaystyle\sum_{i=8}^{k}R([6,2],[i,2]) + R([7,2],[7,2]) \\
\le \displaystyle\sum_{i=8}^{k}\left(\frac{1}{12}i^4 - \frac{241}{12}i^2 + 274i - 1009\right) + 251 \\ \ \\
= \displaystyle\frac{1}{60}k^5 + \frac{1}{24}k^4 - \frac{20}{3}k^3 + \frac{3047}{24}k^2 - \frac{17507}{20}k + 2064.
\end{array}
$$
\end{enumerate}
\end{proof}

\begin{theorem}\label{th3}
\begin{enumerate}[(a)]
\item[]
\item
$R([3,2],[k,1])=2k-1$ for all $k\ge 2$,
\item
$R([4,2],[k,1])\le k^2-22$ for all $k\ge 8$,
\item
$R([5,2],[k,1])\le \frac{1}{3}k^3 + \frac{1}{2}k^2 - \frac{131}{6}k + 95$ for all $k\ge 8$,
\item
$R([6,2],[k,1])\le \frac{1}{12}k^4 + \frac{1}{3}k^3 - \frac{127}{12}k^2 + \frac{505}{6}k - 208$ for all $k\ge 8$,
\item
$R([7,2],[k,1])\le \frac{1}{60}k^5 + \frac{1}{8}k^4 - \frac{10}{3}k^3 + \frac{295}{8}k^2 - \frac{10061}{60}k + 287$ for all $k\ge 8$,
\item
$R([8,2],[k,1])\le \frac{1}{360}k^6 + \frac{1}{30}k^5 - \frac{55}{72}k^4 + \frac{32}{3}k^3 - \frac{11923}{180}k^2 + \frac{2093}{10}k - 239$ for all $k\ge 8$,
\item
$R([9,2],[k,1])\le \frac{1}{2520}k^7 + \frac{1}{144}k^6 - \frac{97}{720}k^5 + \frac{331}{144}k^4 - \frac{12241}{720}k^3 + \frac{2671}{36}k^2 - \frac{20351}{140}k + 24$ for all $k\ge 8$,
\item
$R([10,2],[k,1])\le \frac{1}{20160}k^8 + \frac{1}{840}k^7 - \frac{3}{160}k^6 + \frac{19}{48}k^5 - \frac{3031}{960}k^4 + \frac{4079}{240}k^3 - \frac{200713}{5040}k^2 - \frac{1019}{28}k + 408$ for all $k\ge 8$,
\item
$R([11,2],[k,1])\le \frac{1}{181440}k^9 + \frac{1}{5760}k^8 - \frac{31}{15120}k^7 + \frac{11}{192}k^6 - \frac{3827}{8640}k^5 + \frac{5443}{1920}k^4 - \frac{528539}{90720}k^3 - \frac{9761}{288}k^2 + \frac{965843}{2520}k - 1183$ for all $k\ge 8$.
\end{enumerate}
\end{theorem}

\begin{proof}
For (a), the proof is analogous than the proof of Theorem~\ref{th2}~(a). For the other items, we use the upper bounds $R([4,2],[8,1])\le 42$, $R([5,2],[8,1])\le 123$, $R([6,2],[8,1])\le 300$, $R([7,2],[8,1])\le 657$, $R([8,2],[8,1])\le 1349$, $R([9,2],[8,1])\le 2558$, $R([10,2],[8,1])\le 4722$ and $R([11,2],[8,1])\le 8200$ and the recurrence of Lemma \ref{recur} as follows
$$
R([i,2],[k,1]) \le \sum_{j=9}^{k}R([i-1,2],[j,1]) + R([i,2],[8,1]),
$$
for all integers $k\ge 9$ and for all $i\in\{4,5,\ldots,11\}$. For instance, for $i=4$, we obtain that
$$
\begin{array}{ll}
R([4,2],[k,1]) & \le \sum\limits_{i=9}^k R([3,2],[i,1]) + R([4,2],[8,1])\\
& \le \sum\limits_{i=9}^k (2i-1) + 42=k^2-22.
\end{array}
$$
for all integers $k\ge 9$. The proof for the other values of $i$ is analogous.
\end{proof}

Unfortunately, when $r\ge 3$ (similar as in the classical case) bounds obtained from Theorem \ref{mainth}
(resp. obtained from \eqref{eq2a}, in the classical case) are worse than the bounds obtained from the recursion given in Lemma \ref{recur}
(resp. from the recursion \eqref{eeqq}). 

\section{$\chi_r$-Colorings}\label{xi}

An $r$-edge coloring of $K_n$ is said to be a
{\em $\chi_r$-coloring}, if there exists a labeling of $V(K_n)$ with  $\{1,\dots ,n\}$ and
a function $\phi : \{1,\dots ,n\}\rightarrow  \{0,\dots ,r-1\}$
such that for all $1\le i< j\le n$ the edge $\{i,j\}$  has color $t$ if and only if $\phi(i)=t$.
\medskip

\begin{remark}\label{rem1}
(a) Notice that the value $\phi(n)$ do not play any role in the coloring.
 \smallskip
 
(b) A {\em monochromatic} edge coloring (all edges have the same color $0\le t\le r-1$ ) of $K_n$ is a $\chi_r$-coloring. Indeed, it is enough to take any vertex labeling and to set $\phi(i)=t$ for all $i$.
\smallskip

(c) There exist $r$-edge colorings of $K_n$ that are not $\chi_r$-coloring. For instance, it can be checked that for any labeling of $V(K_3)$ there is not a suitable function $\phi$ giving three different colors to the edges of $K_3$.
\end{remark}

\begin{example}\label{ex2} A 2-coloring of $K_3$ with two edges of the same color and the third one with different color is a  $\chi_2$-coloring. Indeed,  If the edges $\{1,2\}$ and $\{1,3\}$ are colored with color $0$ and the edge $\{2,3\}$ with color $1$ then we take $\phi(1)=0, \phi(2)=1$ and $\phi(3)=1$. 
\end{example}

Let $k\ge 1$ be an integer. Let $\chi_r(k)$ be the smallest integer $n$ such that for any $r$-edge-coloring of $K_N$, $N\ge n$ there exist a clique of order $k$ in which the induced $r$-edge coloring is a $\chi_r$-coloring.  

\begin{remark}\label{rem2}
$\chi_r(k)$ always exists. Indeed, by Ramsey's Theorem, for any $r$-edge coloring of $K_N$, $N\ge R_r(K_k)$ there exist 
a clique order $k$ that is monochromatic which, by Remark \ref{rem1} (b), is a $\chi_r$-coloring. 
\end{remark}

\subsection{$\chi_r$-colorings versus Erd\H os-Rado's colorings}
$\chi_r$-colorings can be considered as a generalization of the classical Ramsey's Theorem. We notice that this generalization is different from the one introduced by Erd\H os and Rado \cite{ER} in which they consider colorings by using an {\em arbitrarily} number of colors (instead of fixing the number of colors $r$) of $[n]\choose k$ according to certain canonical patterns, see also \cite{LR}. Indeed, in the case when $k=2$ the canonical patterns (the edge-colorings of the complete graph) considered by Erd\H os and Rado are those colorings that can be obtained as follows : there exists a (possibly empty) subset $I\subseteq \{1,2\}$  such that the edges $e,f\in {{[n]}\choose 2}$ have the same color if and only if $e_{I}=f_{I}$ where $\{x_1,x_2\}_I=\{x_i\in [n] \ | i\in I\}$. In this case we have the following 4 coloring patterns:
\smallskip

(a) If $I=\emptyset$ then two edges $e,f$ have the same color if and only if $e_{\{\emptyset\}}=\emptyset=f_{\{\emptyset\}}$, that is, all the edges have the same color.
\smallskip

(b) If $I=\{1\}$ then two edges $e,f$ have the same color if and only if $e_{\{1\}}=f_{\{1\}}$, that is,
the smallest element of $e$ is the same as the smallest element of $f$.
\smallskip

(c) If $I=\{2\}$ then two edges $e,f$ have the same color if and only if $e_{\{2\}}=f_{\{2\}}$, that is,
the largest element of $e$ is the same as the largest element of $f$.
\smallskip

(d) If $I=\{1,2\}$ then two edges $e$ and $f$ have the same color if and only if $e_{\{1,2\}}=e=f=f_{\{1,2\}}$, that is, all the edges have different colors. 
\smallskip

Contrary to $\chi_r$-colorings,  the number of colors for Erd\H os-Rado's colorings is not fixed.
So the existence of a Erd\H os-Rado's type coloring do not necessarily implies the existence of a $\chi_r$-coloring.
Nevertheless if the number of colors, say  $r$, is fixed then the patterns (a), (b) and (c) can essentially be considered as $\chi_r$-colorings (it is not the case for pattern (d)). 

\subsection{Values and bounds for $\chi_r(k)$}
We clearly have that $\chi_r(2)=2$.  For $\chi_r(3)$, we first notice that $\chi_r(3) = R_r\left( [3,2] \right)$ and that $K_{[3,2] }$ is a {\em star} $K_{1,2}$(a graph on three vertices, one of degree 2 and two of degree one).  Now, Burr and Roberts \cite{BR} proved that 
$$R\left( K_{1,q_1} , \ldots , K_{1,q_n} \right)=\sum\limits_{j=1}^{n} q_j-n+\epsilon$$
where $\epsilon =1$ if the number of even integers in the set $\{q_1,\dots ,q_n\}$ is even, $\epsilon=2$ otherwise.
Therefore, by applying the above formula when $q_i=2$ for all $i$, we obtain

\begin{equation}\label{r=3} 
\chi_r(3) = \left\{
\begin{array}{ll}
r+1 & \text{for}\ r\ \text{even},\\
r+2 & \text{for}\ r\ \text{odd}.
\end{array}\right.
\end{equation}

\begin{theorem}\label{rec1}
Let $r\ge 2$ be a positive integer and let $k_1,\ldots,k_r$ and $t_1,\ldots,t_r$ be positive integers such that $k_i\ge t_i$ for all $i\in\{1,\ldots,r\}$. Then, 
$$
R([k_1,t_1],\ldots,[k_r,t_r])\le \chi_r\left(\sum_{i=1}^{r}(k_i-t_i-1)+1+\max_{1\le i\le r}\{t_i\}\right).
$$
\end{theorem}

\begin{proof}  Consider a $\chi_r$-coloring of $\displaystyle K_{\chi_r\left(\sum\limits_{i=1}^{r}(k_i-t_i-1)+1+\max_{1\le i\le r}\{t_i\}\right)}$. 
Given the vertex labeling of the $\chi_r$-coloring, 
we consider the complete graph $K'$ induced by the vertices with labels $1,\dots ,\displaystyle\sum_{i=1}^{r}(k_i-t_i-1)+1$ (that is, we remove all the edges induced by the set of vertices $T_1$  with the $\max_{1\le i\le r}\{t_i\}$ largest labels). By the pigeonhole principle,  there is a set $T_2$ of at least $k_i-t_i+1-1$ vertices of $K'$ with the same color for some $i$. Moreover, by definition of $\chi_r$-coloring any edge $\{v_1,v_2\}$ with $v_1\in T_1$ and $v_2\in T_2$ has color $i$, giving the desired monochromatic $K_{[k_i,t_i]}$.
\end{proof}

The following result is an immediate consequence of Theorem \ref{rec1}.

\begin{corollary}\label{rec11} Let $r,k\ge 2$ be integers. Then, 
$$R_r([k,1])\le\chi_r(r(k-2)+2) \text{ and } R_r([k,2])\le\chi_r(r(k-3)+3).$$
\end{corollary}

\begin{proposition}\label{recurrencia} Let $r,k\ge 2$ be integers. Then, 
$$\chi_r(k)\le r\chi_r(k-1)-r+2.$$
\end{proposition}

\begin{proof}
Consider a $r$-edge coloring of $K_{r\chi_r(k-1)-r+2}$ and let $u$ be a vertex. Since $d(u)=r\chi_r(k-1)-r+1$ then there are at least $\left\lceil\frac{r\chi_r(k-1)-r+1}{r}\right\rceil=\chi_r(k-1)$ set of edges with the same color all incident to $u$. Now, by definition of $\chi_r(k-1)$, there is a clique $H$ of order $k-1$ which edge coloring is a $\chi_r$- coloring. So, there is a labeling $\pi$ of $V(H)$, $|V(H)|=k$ and a function $\phi$ giving such coloring. We claim that the $r$-edge coloring of the clique $H'=H\cup u$ is a $\chi_r$-coloring. Indeed, by taking the label $\pi'(i)=\pi(i)+1$ for all vertex $i\neq u$ and $\pi'(u)=1$ and the function $\phi'(1)=1$ and $\phi'(i)=\phi(i-1)$ for each $i=2,\dots ,k$. 
\end{proof}

\begin{proposition}\label{upboundchi}
Let $r,k\ge 2$ be integers. Then, 
$$
\chi_r(k)\le g(k,r)=\left\{\begin{array}{ll}
r^{k-2} + r^{k-3} + \cdots + r^{2} + r +2 = \displaystyle\frac{r^{k-1}-1}{r-1}+1 & \text{for}\ r\ \text{odd},\\
r^{k-2} + r^{k-4} + r^{k-5} + \cdots + r^{2} + r + 2 = \displaystyle\frac{r^{k-3}-1}{r-1} + r^{k-2} + 1 & \text{for}\ r\ \text{even}.
\end{array}\right.
$$
\end{proposition}

\begin{proof}
By equality \eqref{r=3} and by successive applications of Proposition~\ref{recurrencia}.
\end{proof}

\begin{theorem}\label{mainth2}
Let $r\ge 2$ be a positive integer and let $k_1,\ldots,k_r$ and $t_1,\ldots,t_r$ be positive integers such that $k_i\ge t_i$ for all $i\in\{1,\ldots,r\}$. Then, 
$$
R([k_1,t_1],\ldots,[k_r,t_r])\le g(k,r)
$$
where
$$
k:=\sum_{i=1}^{r}(k_i-t_i-1)+1+\max_{1\le i\le r}\{t_i\}.
$$
\end{theorem}

\begin{proof}
By Theorem~\ref{rec1} and Proposition~\ref{upboundchi}.
\end{proof}

We believe that the above upper bound for $R_r([k,1])$ is smaller than the one given by Corollary \ref{cor1a} (see equation \eqref{eq2}) for some values of $k$.

\begin{conjecture}\label{conj1} Let $r\ge 3$ be an integer. Then, for all $3\le k\le r^{3/2}+r-1$ 
$$g((r(k-2)+2,r)<\binom{r(k-1)}{k-1,k-1,\ldots\ldots,k-1}= \frac{(rk-r)!}{((k-1)!)^r}\cdot$$
\end{conjecture}

We have checked the validity of the above conjecture for all $3\le r\le 150$ by computer calculations. 
Conjecture \ref{conj1} is also supported graphically, by considering the continual behaviour of 

$$f(k,r)= g((r(k-2)+2,r)- \frac{(rk-r)!}{((k-1)!)^r}\cdot$$

To see that, we may use the fact that  $\Gamma(z+1)=z!$ when $z$ is a nonnegative integer, obtaining 

$$f(k,r)= g((r(k-2)+2,r)- \frac{\Gamma(r(k-1)+1)}{\Gamma^r(k)}$$

where $\Gamma(z) $ is the well-known {\em gamma} function\footnote{The gamma function is defined as $\Gamma(z)=\int_0^{+\infty}t^{z-1}e^{-t} dt$ for any $z\in\mathbb{C}$ with $Re(z)>0$. Moreover,  $\Gamma(z+1)=z!$ when $z$ is a nonnegative integer.}, see Figure \ref{fig1}.
  
\begin{figure}[htb] 
\includegraphics[width=.42\textwidth]{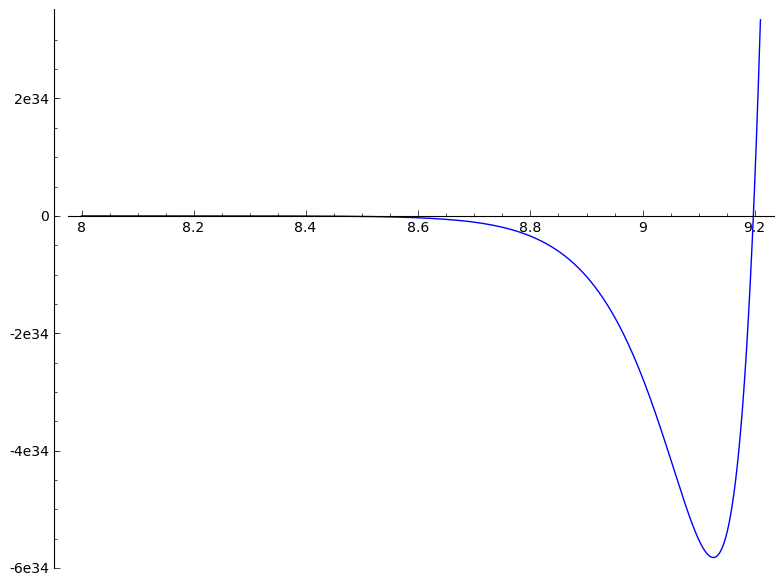}\ \ 
\includegraphics[width=.42\textwidth]{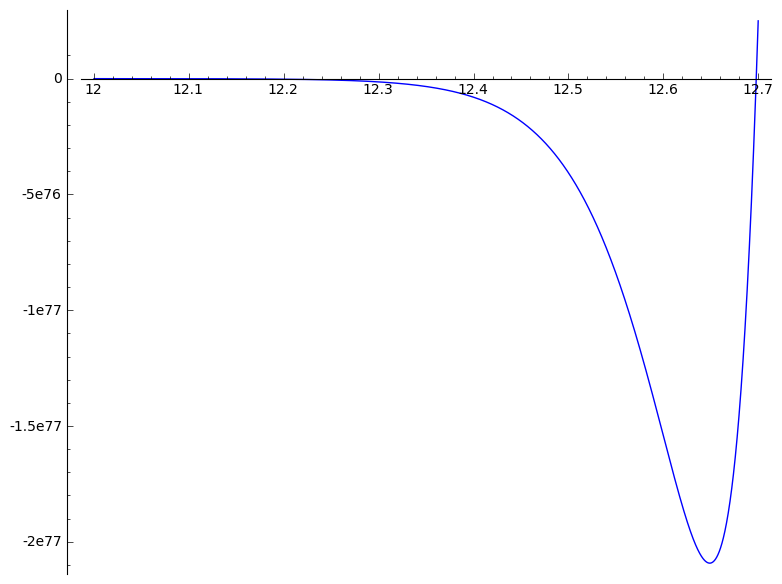}
\caption{Behaviours of $f(4,k)$ with $8\le k<10$ (left) and $f(5,k)$ with $12\le k<13$ (right). We notice that due to the scaling used in the figures (in order to plot the minimum) the function $f$ seems very close to zero but in fact it is very far apart, $f(4,8)\le -1,8\times 10^{29}$ for the left one and $f(5,12)\le -5,7\times 10^{72}$ for the right one.}
\label{fig1}\end{figure}

We have also checked (by computer) that for each  $3\le r\le 150$ there is an interval $I_r$ (increasing as $r$ is growing) such that for each $k\ge 3, k\in I_r$ the function  $g(r(k-3)+3,r)$ (resp. $g(r(k-4)+4,r)$) is a smaller upper bound for $R_r([k,2])$ (resp. for $R_r([k,3])$) than the  corresponding ones obtained from Corollary \ref{cor1a}. In view of the latter, we pose the following 

\begin{question} Let $t\ge 1$ and $r\ge 3$ be integers. Is there a function $c(r)$ such that for all $3\le k\le c(r)$ 
$$g(r(k-t)+t,r)<t \binom{r(k-t)}{k-t,k-t,\ldots\ldots,k-t} \ ?$$
\end{question}

\end{document}